\documentclass[reqno, 11pt]{amsart}
\usepackage{amssymb}
\usepackage[numbers]{natbib}
\pdfoutput=1

\usepackage[nesting]{hyperref}

\usepackage[pdftex]{graphicx}
\usepackage{listings}
\usepackage{multirow}
\usepackage{placeins}
\usepackage{color}
\usepackage{subfigure}
\usepackage{lscape}

\newcommand{\BlackBoxes}{\global\overfullrule5pt}

\BlackBoxes

\newcommand{\R}{\mathbb{R}}
\newcommand{\N}{\mathbb{N}}

\newcommand{\E}{\mathbb{E}}

\newcommand{\PP}{\mathbb{P}}
\newcommand{\Eop}{\operatorname{\mathbb{\E}}}

\newcommand{\Pop}{\operatorname{\mathbb{\PP}}}

\newtheorem{theorem}{Theorem}
\newtheorem{corollary}[theorem]{Corollary}

\newtheorem{proposition}[theorem]{Proposition}
\theoremstyle{definition}

\newtheorem{remark}[theorem]{Remark}
\newtheorem{definition}[theorem]{Definition}
\numberwithin{equation}{section} \numberwithin{theorem}{section}

\def\0{\kern0pt\-\nobreak\hskip0pt\relax}

\makeatletter
\AtBeginDocument{%
 \def\@serieslogo{%
 \vbox to\headheight{%
 \parindent\z@ \fontsize{6}{7\p@}\selectfont
 \vss}}}

\def\makeoverbar#1#2#3#4#5#6#7{%
 \setbox0=\hbox{$\m@th#2\mkern#5mu{{}#3{}}\mkern#6mu$}%
 \setbox1=\null \dimen@=#4\fontdimen8#13 \dimen@=3.5\dimen@
 \advance\dimen@ by \ht0 \dimen@=-#7\dimen@ \advance\dimen@ by \wd0
 \ht1=\ht0 \dp1=\dp0 \wd1=\dimen@
 \dimen@=\fontdimen8#13 \fontdimen8#13=#4\fontdimen8#13
 \rlap{\hbox to \wd0{$\m@th\hss#2{\overline{\box1}}\mkern#5mu$}}
 \fontdimen8#13=\dimen@}

\def\mylabel#1#2{{\def\@currentlabel{#2}\label{#1}}}

\makeatother

\begin{document}


\makeatletter \providecommand\@dotsep{5} \makeatother

\title[Partially Observable Risk-Sensitive Markov Decision Processes]{Partially Observable Risk-Sensitive Markov Decision Processes}

\author[N. \smash{B\"auerle}]{Nicole B\"auerle${}^*$}
\address[N. B\"auerle]{Institute for Stochastics,
Karlsruhe Institute of Technology, D-76128 Karlsruhe, Germany}

\email{\href{mailto:nicole.baeuerle@kit.edu}
{nicole.baeuerle@kit.edu}}


\author[U. \smash{Rieder}]{Ulrich Rieder$^\ddag$}
\address[U. Rieder]{University of Ulm, D-89069 Ulm, Germany}

\email{\href{mailto:ulrich.rieder@uni-ulm.de} {ulrich.rieder@uni-ulm.de}}

\maketitle

\begin{abstract}
We consider the problem of minimizing a certainty equivalent of the total or discounted cost over a finite and an infinite time horizon which is generated by a Partially Observable Markov Decision Process (POMDP). The certainty equivalent is defined by $U^{-1}(\Eop U(Y))$ where $U$ is an increasing function. In contrast to a risk-neutral decision maker, this optimization criterion takes the variability of the cost into account. It contains as a special case the classical risk-sensitive optimization criterion with an exponential utility.   We show that this optimization problem can be solved by embedding the problem into a completely observable Markov Decision Process with extended state space and give conditions under which an optimal policy exists. The state space has to be extended by the joint conditional distribution of current unobserved state and accumulated cost. In case of an exponential utility, the problem simplifies considerably and we rediscover what in previous literature has been named {\em information state}. However, since we do not use any change of measure techniques here, our approach is simpler. A simple example, namely a risk-sensitive Bayesian house selling problem is considered to illustrate our results.

\end{abstract}

\vspace{0.5cm}
\begin{minipage}{14cm}
{\small
\begin{description}
\item[\rm \textsc{ Key words}]
{\small Partially Observable Markov Decision Problem, Certainty Equivalent, Exponential Utility, Updating Operator, Value Iteration.}
\end{description}
}
\end{minipage}

\section{Introduction}\label{sec:intro}\noindent
In this work we consider Partially Observable Markov Decision Processes (POMDP) under a general risk-sensitive optimization criterion for problems with finite and infinite time horizon. This is a continuation of our research published in \cite{br14}. More precisely our aim is to minimize the certainty equivalent of the accumulated total cost of a POMDP. In case of an infinite time horizon, costs have to be discounted. The certainty equivalent of a random variable is defined by $U^{-1}(\Eop U(X))$ where $U$ is an increasing function. If $U(x)=x$ we obtain as a special case the classical risk-neutral decision maker. The case $U(x)=\frac1\gamma e^{\gamma x }$ is often referred to as 'risk-sensitive', however the risk-sensitivity is here only expressed in a special way through the risk-sensitivity parameter $\gamma\neq 0$.  More general, the certainty equivalent may be written (assuming enough regularity of $U$) as
\begin{equation}\label{eq:Urep} U^{-1}\Big(\Eop\big[U(X)\big]\Big)\approx \Eop X - \frac12 l_U(\Eop X) Var[X]\end{equation}
where $$ l_U(x) = -\frac{U''(x)}{U'(x)}$$ is the {\em Arrow-Pratt} function of absolute risk aversion. In case of an exponential utility, this absolute risk aversion is constant (for a discussion see \cite{bpli03}). If $U$ is concave, the variance is subtracted and the decision maker is risk seeking in case cost is minimized, if $U$ is convex, then the variance is added and the decision maker is risk averse.

In case of complete observation it has been shown in \cite{br14} that this problem can be recast in the theory of Markov Decision Processes (MDP) by enlarging the state space with the total discounted cost that has been incurred so far. Numerical solution procedures via linear programming of these completely observable general risk-sensitive Markov Decision Processes can be found in \cite{hj14}. The average cost version of this problem is treated in \cite{cchh15} and for an application in insurance see \cite{bj14}. Now we assume that only one of two components of a controlled Markov process can be observed. However, also the cost may depend on both components which leads to the situation that the cost incurred so far is an unobservable quantity. It is well-known that in case of a risk-neutral decision maker, the partially observable problem can be solved by a completely observable MDP when we enlarge the state space by the conditional distribution of the unobservable state, given the observable history of the process (see e.g. \cite{br11} chapter 5, \cite{herler89} chapter 4 or \cite{hin70} chapter 7). As far as the risk-sensitive problem is concerned we proceed in a similar way. This time however, the corresponding problem with complete observation possesses already an enlarged state consisting of the process state and the total discounted cost so far. Thus, to cope with the partially observable model we construct a Markov Decision Process where the state consists of the observable part of the state and the joint conditional distribution of the unobservable part of the state and unobservable total cost so far, given the observable history of the process.

Early papers \cite{rhe74,yush34} provided rigorous mathematical treatment of POMDPs with Borel state and action spaces. These references already present the solution procedure via the enlargement of the state space and the reduction to an ordinary Markov Decision Process. For a detailed discussion of the theory in the classical risk-neutral setting and for several applications see \cite{br11} chapter 5 and \cite{herler89} chapter 4.
Risk-sensitive Markov Decision processes with the exponential utility have been discussed intensively since the seminal paper of \cite{hm72}. For further references we refer the reader to \cite{br14}. Recent applications of this criterion in a wide range of portfolio optimization problems can be found in \cite{dl14}.
Papers which combine the exponential utility with POMDPs are among others \cite{jjr94,hh99,dims99,s04,cchh05}. In all these papers a control model formulation has been used, where the true, unobservable and controlled state process is a Markov process (under Markovian policies) and observations are obtained by perturbed signals of this process. A change of measure technique is used to obtain independent signals. In order to apply MDP theory, the state space has been enlarged by a quantity that has been called an 'information vector'. In the present paper we use a more general model formulation where both parts (observable and unobservable state) are jointly Markovian and can be controlled jointly. This setting also covers the Bayesian case where the unknown state part is simply an unknown parameter. Also note that our optimization criterion is not restricted to the exponential utility and we do not need a change of measure technique to derive our filter. Moreover, the general approach implies a very natural interpretation for the 'information vector' in the exponential utility case. Besides \cite{jjr94} all the previously mentioned papers focus on the risk-sensitive average criterion by using the vanishing discount approach, i.e., by looking at the $\beta$-discounted problem and by letting $\beta$ go to $1$. In \cite{cchh05} a finite state and action space is considered and emphasis is laid on numerical aspects of the problem. A discrete-time linear quadratic risk-sensitive stochastic control problem with incomplete state information is solved in \cite{whi90}.

Our paper is organized as follows: In the next section we introduce the underlying POMDP and define general history-dependent (deterministic) policies for this model. In section \ref{sec:finiteh} we consider the finite horizon general risk-sensitive problem and introduce continuity and compactness assumptions which will guarantee the existence of optimal policies. Then the problem is embedded into a suitably defined Markov Decision Process where the state space contains among others a joint conditional distribution of the unobservable state and total accumulated cost so far, given the observed process. An updating-operator is defined to create a forward iteration of this joint conditional distribution.  The main theorem of this section (Theorem \ref{theo:Bellman1}) states the validity of the embedding procedure and the existence of optimal policies. Section \ref{sec:special} contains some important special cases. Among them the situation where the cost function does not depend on the unobservable state in which case the updating operator simplifies to the updating operator for classical risk-neutral POMDPs. In case the exponential utility function is used, we rediscover some results of the previous literature. We also consider the case of a power utility where we only get a slight simplification. In Section \ref{sec:casino} we consider a simple risk-sensitive Bayesian house selling problem. We prove the existence of so-called 'reservation levels' which can be seen as thresholds for the acceptance of an offer. These reservation levels depend only on the conditional distribution. In the last section we consider the problem with infinite time horizon and distinguish the case of a convex and a concave utility functions which require separate proofs due to different inequlities. The main theorems (Theorem \ref{theo:limit2}, Theorem \ref{theo:limit4}) show that the value function of the problem can be obtained from a fixed point equation and that an optimal policy exists which is not stationary but still can be generated by only one decision function.

\section{General Partially Observable Risk-Sensitive Markov Decision Processes}\label{sec:intro}\noindent
We suppose that a {\em partially observable Markov Decision Processes} is given which we introduce as follows: We denote this process by $(X_n,Y_n)_{n\in\N_0}$ and assume that  the state space is $E_X\times E_Y$ where $E_X$ and $E_Y$ are Borel spaces, i.e., Borel subsets of some Polish spaces. The $x$-component will be the {\em observable} part, the $y$-component {\em cannot be observed} by the controller. Actions can be taken from a set $A$ which is again a Borel space. The set $D\subset E_X\times A$ is a Borel subset of $E_X\times A$. By $D(x) := \{a\in A : (x,a)\in D\}$ we denote the feasible actions depending on the observable state part $x$. We assume that $D$ contains the graph of a measurable mapping from $E_X$ to $A$. There is a stochastic transition kernel $Q$ from $D\times E_Y$ to $E_X\times E_Y$ which determines the distribution of the new state pair given the current state and action. So $Q(B|x,y,a)$ is the probability that the next state pair is in $B\in\mathcal{B}(E_X\times E_Y)$, given the current state is $(x,y)$ and action $a\in D(x)$ is taken. In what follows we assume that the transition kernel $Q$ has a measurable density $q$ with respect to some $\sigma$-finite measures $\lambda$ and $\nu$, i.e.,
$$ Q(B | x,y,a) =\int_{B} q(x',y'|x,y,a) \lambda(dx') \nu(dy'),\quad B\in\mathcal{B}(E_X\times E_Y).$$
For convenience we introduce the marginal transition kernel density by
$$ q^X(x'|x,y,a) := \int_{E_Y} q(x',y'|x,y,a) \nu(dy').$$
We assume that the initial distribution $Q_0$ of $Y_0$ is known. Further we have a measurable one-stage cost function $c:D\times E_Y\to \R_+$. We assume in particular that the cost $c(x,y,a)$ also depends on the unknown state part $y$. Finally we have a discount factor $\beta \in (0,1]$.

Next we introduce policies for the controller. Here it is important to consider the {\em set of observable histories} which are defined as follows:
\begin{eqnarray*}
  H_0 &:=& E_X \\
  H_n &:=& H_{n-1}\times A\times E_X.
\end{eqnarray*}
An element $h_n=(x_0,a_0,x_1,\ldots,x_n)\in H_n$ denotes the observable history of the process up to time $n$.

\begin{definition}
\begin{enumerate}
\item[a)] A measurable mapping $g_n: H_n\to A$ with the property $g_n(h_n) \in D(x_n)$ for
$h_n\in H_n$ is called a {\em decision rule} at stage $n$.
\item[b)] A sequence $\pi=(g_0,g_1,\ldots)$ where $g_n$ is a decision rule at stage $n$
 for all $n$, is called {\em policy}. We denote by $\Pi$ the set of all policies.
\end{enumerate}
\end{definition}

\section{Finite Horizon Problems}\label{sec:finiteh}
In this section we consider problems with finite time horizon $N$.
For a fixed policy $\pi=(g_0,g_1,\ldots)\in\Pi$ and fixed (observable) initial state $x\in E_X$, the initial
distribution $Q_0$ together with the transition kernel $Q$ define by a theorem of Ionescu Tulcea a probability measure $\Pop^\pi_{xy}$ on
$(E_X\times E_Y)^{N+1} $ endowed with the product $\sigma$-algebra. More precisely $\Pop_{xy}^\pi$ is the probability measure under policy $\pi$ given $X_0=x$ and $Y_0=y$. Later we also use the probability measure $\Pop_x^\pi(\cdot) := \int \Pop_{xy}^\pi (\cdot)Q_0(dy)$. For $\omega=(x_0,y_0,\ldots,x_{N},y_{N})\in (E_X\times E_Y)^{N+1}$ we define the random variables $X_n$ and $Y_n$ in a canonical way by
their projections $$ X_n(\omega)=x_n,\quad Y_n(\omega)=y_n.$$

If $\pi=(g_0,g_1,\ldots )\in\Pi$ is a given policy, we define recursively
\begin{eqnarray*}
  A_0 &:=& g_0(X_0) \\
  A_n &:=& g_n(X_0,A_0,X_1,\ldots ,X_n),
\end{eqnarray*}
the sequence of actions which are chosen successively under policy $\pi$.
We assume that the decision maker is risk averse and has a utility function $U:\R_+\to\R$ which is continuous and strictly increasing.
The optimization problem is defined as follows. For $\pi\in\Pi$ and $X_0=x$ denote
$$ J_{N\pi}(x) := \int_{E_Y} \Eop_{xy}^\pi \left[ U\Big(\sum_{k=0}^{N-1}\beta^k c(X_k,Y_k,A_k)\Big)\right]Q_0(dy)$$ and
\begin{equation}\label{eq:unobservedMDP} J_N(x) := \inf_{\pi\in\Pi} J_{N\pi}(x). \end{equation}
Note that in case $U(x)=x$ we end up with the usual risk neutral Partially Observable Markov Decision Process setup (see e.g. \cite{br11} chapter 5, \cite{herler89} chapter 4). Here however, if $U$ is strictly concave, then $U$ is a utility function  and $U^{-1}(J_N(x))$ represents a {\em certainty equivalent}.
If $U$ is concave, we can see from \eqref{eq:Urep} that the decision maker is risk seeking and if $U$ is convex, then the decision maker is risk averse.

In what follows we show how to solve these kind of problems by using an {\em embedding technique}. In order to later ensure  the existence of integrals and optimal policies we make the following assumptions (A):
\begin{itemize}
  \item[(i)] $U:[0,\infty)\to\R$ is continuous and strictly increasing,
  \item[(ii)] $D(x)$ is compact for all $x\in E_X$,
  \item[(iii)] $x\mapsto D(x)$ is upper semicontinuous, i.e. for all $x\in E_X$ it holds: If $x_n\to x$ and $a_n\in D(x_n)$ for all $n\in\N$, then $(a_n)$ has an accumulation point in $D(x)$,
  \item[(iv)] $(x,y,a) \mapsto c(x,y,a)$ is continuous,
  \item[(v)] $(x,y,x',y',a) \mapsto q(x',y'|x,y,a)$ is continuous and bounded.
    \item[(vi)] $c$ is bounded, i.e., there exist constants  $0<\underline{c}<\overline{c}$ with $\underline{c}\le c(x,y,a) \le \overline{c}$.
\end{itemize}

\begin{remark}
Note that these assumptions are quite strong, however include in particular the case when state and action spaces are finite. (A)(ii-v) also ensure the existence of optimal policies for  risk-neutral POMDP.  
\end{remark}

In \cite{br14} we have solved problem \eqref{eq:unobservedMDP} for the observable case by extending the state space to include the accumulated cost so far. Now in the unobservable model, the state $y$ {\em and} the accumulated cost so far cannot be observed because it depends on $y$. Thus, we proceed as in risk-neutral POMDPs (see e.g. \cite{rhe74,yush34}) and consider probability measures $\mu$ on $E_Y\times \R_+$:
\begin{eqnarray*}
  \mu  &\in& \Pop_b(E_Y\times \R_+) := \Big\{ \mu \;\mbox{ is a probability measure on the}\; \sigma\mbox{-algebra }\; \mathcal{B}(E_Y\times \R_+)\;\mbox{such}\\ && \hspace*{3cm}  \mbox{that there exists a constant}\; K=K(\mu) >0 \;\mbox{with}\; \mu(E_Y\times [0,K])=1\Big\}.
\end{eqnarray*}
$\mu$ plays the role of the conditional distribution on the larger state space of hidden state component and accumulated cost. The precise interpretation will be seen in Theorem \ref{theo:muinterpr}.
In order to solve the optimization problem, we need, as in the risk-neutral case, an updating procedure for the conditional distributions which generates the filter process. The following updating-operator $\Psi:E_X\times A\times E_X\times  \Pop_b(E_Y\times \R_+)\times \R_+\to \Pop_b(E_Y\times \R_+)$ will do the task:
 \begin{eqnarray}\label{eq:Psi}\Psi(x,a,x',\mu,z) (B) &:=&
 \frac{\int\limits_{E_Y}\int\limits_{\R_+} \Big(\int\limits_{B} q(x',y'|x,y,a) \nu(dy')\delta_{s+z c(x,y,a)}(ds')\Big) \mu(dy,ds) }{\int_{E_Y} q^X(x'|x,y,a) \mu^Y(dy) }
 \end{eqnarray}
where $B\in \mathcal{B}(E_Y\times \R_+)$ and $\mu^Y(dy) := \mu(dy,\R_+)$ is the $Y$-marginal distribution of $\mu$. Later we will also need the $S$-marginal  $\mu^S(ds) := \mu(E_Y, ds)$. We define the updating operator only when the denominator is positive.
For $n\in\N$, $h_n := (x_0,a_0,\ldots ,x_n)$ and $B\in \mathcal{B}(E_Y\times \R_+)$  define now a sequence of probability measures
\begin{eqnarray}
\nonumber  \mu_0(B|h_0) &:=& (Q_0\otimes \delta_0)(B), \\
\label{eq:murec1}  \mu_{n+1}(B|h_{n},a,x')&:=& \Psi\big(x_n,a,x',\mu_n(\cdot|h_n),\beta^n\big)(B).
\end{eqnarray}

The next theorem shows that the sequence of probability measures $(\mu_n)$ has the intended interpretation. For this purpose define the r.v.
$$ S_0 := 0,\quad S_n := \sum_{k=0}^{n-1}\beta^k c(X_k,Y_k,A_k),\quad n\in\N.$$
We then obtain:

\begin{theorem}\label{theo:muinterpr}
Suppose $(\mu_n)$ is given by the recursion \eqref{eq:murec1}. For $n\in\N_0$ and all $\pi\in\Pi$ it holds that
$$  \Pop_x^\pi\big((Y_n,S_n)\in B|X_0,A_0,\ldots,X_n\big) = \mu_n(B|X_0,A_0,\ldots,X_n) \; \Pop_x^\pi-a.s.,\quad\mbox{for } B\in \mathcal{B}(E_Y\times\R_+).$$
\end{theorem}

\begin{proof} Recall that $\Pop_x^\pi(\cdot) := \int \Pop_{xy}^\pi (\cdot)Q_0(dy)$.
We first show that
\begin{equation}\label{eq:filter2} E_x^\pi\Big[ v(X_0,A_0,X_1,\ldots ,X_n,Y_n,S_n)\Big] = E_x^\pi\Big[ v'(X_0,A_0,X_1,\ldots ,X_n)\Big] \end{equation}
for all bounded and measurable $v: H_n\times E_Y \times \R_+ \to \R$ and
$$v'(h_n) := \int_{E_Y}\int_{\R_+} v(h_n,y_n,s_n) \mu_n(dy_n,ds_n|h_n).$$ We do this by induction. For $n=0$ both sides reduce to
$\int v(x,y,0) Q_0(dy).$ Now suppose the statement is true for $n-1$. We simply write $g_n$ instead of $g_n(h_n)$.  We obtain for the left-hand side with a given observable history $h_{n-1}$:
\begin{eqnarray*}
 && E_x^\pi\Big[ v(h_{n-1},A_{n-1},X_n,Y_n,S_n)\Big]  = \int_{E_Y}\int_{\R_+} \mu_{n-1}(dy_{n-1},ds_{n-1}|h_{n-1})  \\
 && \hspace*{2cm} \cdot \int_{E_Y}\int_{E_X} \nu(dy_n) \lambda(dx_n) q(x_n,y_n|x_{n-1},y_{n-1},g_{n-1})  \\
   && \hspace*{2cm} \cdot \int_{\R_+} \delta_{s_{n-1}+\beta^{n-1}c(x_{n-1},y_{n-1},g_{n-1})} (ds_n) v(h_{n-1},g_{n-1},x_n,y_n,s_n)  \\
   &=&  \int_{E_Y}\int_{\R_+} \mu_{n-1}(dy_{n-1},ds_{n-1}|h_{n-1}) \int_{E_Y}\int_{E_X} \nu(dy_n) \lambda(dx_n) q(x_n,y_n|x_{n-1},y_{n-1},g_{n-1})  \\
   &&  \hspace*{2cm} \cdot v\big(h_{n-1},g_{n-1},x_n,y_n,s_{n-1}+\beta^{n-1}c(x_{n-1},y_{n-1},g_{n-1})\big).
\end{eqnarray*}
For the right-hand side we obtain (where we insert the recursion for $\mu_n$ in the third equation and use Fubini's theorem, so that the normalizing constant of $\mu_n$ cancels out):
\begin{eqnarray*}
 && E_x^\pi\Big[ v'(h_{n-1},A_{n-1},X_n)\Big]  = \int_{E_Y}\int_{\R_+} \mu_{n-1}(dy_{n-1},ds_{n-1}|h_{n-1})  \\
  && \hspace*{2cm} \cdot \int_{E_X} \lambda(dx_n) q^X(x_n|x_{n-1},y_{n-1},g_{n-1})
    v'(h_{n-1},g_{n-1},x_n) \\
   &=&  \int_{E_Y} \mu_{n-1}^Y(dy_{n-1}|h_{n-1}) \int_{E_X} \lambda(dx_n) q^X(x_n|x_{n-1},y_{n-1},g_{n-1}) \\
   && \hspace*{2cm} \cdot \int_{E_Y}\int_{\R_+} \mu_n(dy_n,ds_n|h_n) v(h_{n-1},g_{n-1},x_n,y_n,s_{n})  \\
   &=& \int_{E_Y}\int_{E_X}\nu(dy_n) \lambda(dx_n) \int_{E_Y}\int_{\R_+}\mu_{n-1}(dy_{n-1},ds_{n-1}|h_{n-1}) q(x_n,y_n|x_{n-1},y_{n-1},g_{n-1})\\
   && \hspace*{2cm}   \cdot \int_{\R_+}  \delta_{s_{n-1}+\beta^{n-1}c(x_{n-1},y_{n-1},g_{n-1})} (ds_n)  v(h_{n-1},g_{n-1},x_n,y_n,s_{n}) \\
 &=&  \int_{E_Y}\int_{\R_+} \mu_{n-1}(dy_{n-1},ds_{n-1}|h_{n-1}) \int_{E_Y}\int_{E_X} \nu(dy_n) \lambda(dx_n) q(x_n,y_n|x_{n-1},y_{n-1},g_{n-1})  \\
   &&  \hspace*{2cm} \cdot v\big(h_{n-1},g_{n-1},x_n,y_n,s_{n-1}+\beta^{n-1}c(x_{n-1},y_{n-1},g_{n-1})\big).
\end{eqnarray*}
Thus equation \eqref{eq:filter2} is proved. It implies in particular for $v=1_{B\times C}$ with $B\in \mathcal{B}(E_Y\times\R_+)$ and $C\subset E_X\times A\times \ldots \times E_X$ a measurable set of histories  until time $n$ that
$$ \Pop_x^\pi\big((Y_n,S_n)\in B, (X_0,A_0,\ldots,X_n)\in C\big) = \Eop_x^\pi\big[ \mu_n(B|X_0,A_0,\ldots,X_n) 1_C((X_0,A_0,\ldots,X_n))\big].$$
This in turn yields by definition that $\mu_n(B|X_0,A_0,\ldots,X_n) $ is a conditional $\Pop_x^\pi$-distribution of $(Y_n,S_n)$ given the history $(X_0,A_0,\ldots,X_n)$.
\end{proof}

Now we turn again to the optimization problem \eqref{eq:unobservedMDP}. Motivated by the previous result we define for $x\in E_X$, $\mu\in \Pop_b(E_Y\times \R_+)$, $z\in(0,1]$ and $n=1,\ldots,N$:
\begin{eqnarray}
\label{eq:Vpi-def}  V_{n\pi}(x,\mu,z)  &:=&  \int_{E_Y}\int_{\R_+} \Eop_{xy}^\pi \left[ U\Big(s+z \sum_{k=0}^{n-1}\beta^k c(X_k,Y_k,A_k)\Big)\right]\mu(dy,ds) \\
\label{eq:V-def}  V_{n}(x,\mu,z)  &:=& \inf_{\pi \in \Pi}  V_{n\pi}(x,\mu,z).
\end{eqnarray}
Obviously we have that $J_N(x) = V_N(x,Q_0\otimes \delta_0,1)$ where $\delta_x$ is the Dirac-measure at the point $x\in\R$.
However, problem \eqref{eq:V-def} can be solved with the general theory of POMDP and \cite{br14} by defining a suitable MDP. For this purpose let us define for a probability measure $\mu \in \Pop(E_Y)$
\begin{eqnarray*}
   Q^X(B|x,\mu,a) &:=& \int_B\int_{E_Y}  q^X(x'|x,{y},a)   \mu(dy) \lambda(dx'),\; B\in\mathcal{B}(E_X)\\
\end{eqnarray*}

We consider a Markov Decision Process with state space ${E}:=E_X\times \Pop_b(E_Y\times \R_+)\times (0,1]$, action space $A$ and admissible actions given by the set $D$. The one-stage cost is zero and the terminal cost function is $ V_0(x,\mu,z):= \int\int U(s) \mu(dy,ds) $. Note that for all $\mu \in \Pop_b(E_Y\times \R_+)$ the expectation is well-defined since the support of $\mu$ in the $s$-component is a compact set. The transition law is given by $\tilde{Q}(\cdot|x,\mu,z,a)$ which is for $(x,\mu,z,a) \in E\times A$, $a\in D(x)$ and a measurable subset $B\subset E$ defined by
\begin{eqnarray*}
  \tilde{Q}(B|x,\mu,z,a):=\int_{E_X} 1_B\Big((x',\Psi(x,a,x',\mu,z),\beta z)\Big) Q^X(dx'|x,\mu^Y,a).
\end{eqnarray*}
Note that $\tilde{Q}$ is again a transition kernel.
Decision rules in the MDP setting are given by measurable mappings $f:{E}\to A$ such that $f(x,\mu,z)\in D(x)$. We denote by $F$ the set of decision rules and by $\Pi^M$ the set of Markov policies $\pi=(f_0,f_1,\ldots)$ with $f_n\in F$. Note that `Markov' refers to the fact that the decision at time $n$ depends only on $x,\mu$ and $z$. Further note that we have $\Pi^M\subset \Pi$ in the following sense: For every $\pi=(f_0,f_1,\ldots)\in \Pi^M$ we find a $\sigma=(g_0,g_1,\ldots)\in\Pi$ such that
\begin{eqnarray*}
 g_0(x_0) &:=& f_0(x_0,\mu_0,1),\\
 g_n(h_n)  &:=& f_n\big(x_n,\mu_n(\cdot|h_n),\beta^n\big),\; n\in \N.
\end{eqnarray*}
With this interpretation $V_{n\pi}$ is also defined for $\pi\in\Pi^M$.

Let us now introduce the set
\begin{eqnarray*} \mathcal{C}({E}) &:=& \Big\{ v:{E}\to\R : v\;\mbox{is lower semicontinuous and}\; v\ge V_0\Big\},
\end{eqnarray*}
where we use the topology of weak convergence on $\Pop_b(E_Y\times \R_+)$.
For $v\in \mathcal{C}({E})$ and $f\in F$ we consider the operator
$$ (T_fv)(x,\mu,z) := \int_{E_X} v\Big(x',\Psi(x,f(x,\mu,z),x',\mu,z),\beta z\Big) Q^X\big(dx'|x,\mu^Y,f(x,\mu,z)\big) ,\quad (x,\mu,z)\in {E}$$
which is well-defined. The minimal cost operator of this Markov Decision Model is given by
\begin{equation}\label{eq:Top}(T v)(x,\mu,z) := \inf_{a\in D(x)} \int_{E_X} v\Big(x',\Psi(x,a,x',\mu,z),\beta z\Big) Q^X(dx'|x,\mu^Y,a),\quad (x,\mu,z)\in {E}\end{equation}
which is again well-defined and $T_f V_0 \ge TV_0\ge V_0$ (see also the proof below). Note that $V_0\in \mathcal{C}({E})$. If a decision rule $f\in F$ is such that $T_fv=Tv$, then $f$ is called a {\em minimizer} of $v$. We obtain:

\begin{theorem}\label{theo:Bellman1}
It holds  that
\begin{itemize}
  \item[a)] For a policy $\pi=(f_0,f_1,f_2,\ldots)\in\Pi^M$ we have the following cost iteration:\\  $V_{n\pi} = T_{f_0}\ldots T_{f_{n-1}} V_{0}$ for $n=1,\ldots ,N$.
  \item[b)] $V_n\in \mathcal{C}({E})$ and $V_{n} = T V_{n-1}$, for $n=1,\ldots ,N$, i.e.,
  $$ V_{n+1}(x,\mu,z) = \inf_{a\in D(x)} \int_{E_X} V_{n}\Big(x',\Psi(x,a,x',\mu,z),\beta z\Big) Q^X(dx'|x,\mu^Y,a),\; (x,\mu,z)\in {E}.$$
The value function of \eqref{eq:unobservedMDP} is then given by $J_N(x)=V_N(x,Q_0\otimes \delta_0,1)$.
  \item[c)] For every $n=1,\ldots ,N$ there exists a  minimizer $f^*_n\in F$ of $V_{n-1}$ and $(g_0^*,\ldots,g_{N-1}^*)$ with
 $$ g_n^*(h_n) := f_{N-n}^*\big( x_n, \mu_n(\cdot|h_n),\beta^n\big),\quad n=0,\ldots,N-1$$
is an optimal policy for problem \eqref{eq:unobservedMDP}. Note that the optimal policy consists of decision rules which depend on the current state and the current joint conditional distribution of accumulated cost and hidden state.
  \end{itemize}
\end{theorem}

\begin{proof}
  The proof of part a) is by induction. For $n=1$ we obtain with $a:=f_0(x,\mu,z)$:
\begin{eqnarray*}
    T_{f_0} V_{0}(x,\mu,z) &=& \int_{E_X}  V_{0}\Big(x',\Psi(x,a,x',\mu,z),\beta z\Big) Q^X(dx'|x,\mu^Y,a)\\
   &=& \int_{E_Y}\int_{\R_+}\int_{E_X} \int_{\R_+} U(s')  \delta_{s+z c(x,y,a)}(ds') q^X(x'|x,y,a)\lambda(dx')\mu(dy,ds)\\
    &=& \int_{E_Y} \int_{\R_+} U\big(s+z c(x,y,a)\big) \mu(dy,ds)\\
    &=& V_{1\pi}(x,\mu,z).
\end{eqnarray*}

Suppose the statement is true for $V_{n\pi}$. In order to ease notation we denote for a policy $\pi=(f_0,f_1,f_2,\ldots)\in\Pi^M$ by $\vec{\pi}=(f_1,f_2,\ldots)$ the shifted policy. Moreover let again $a := f_0(x,\mu,z)$. Then
\begin{eqnarray*}
  &&(T_{f_0}\ldots T_{f_{n-1}} V_{0})(x,\mu,z) = \int_{E_X} V_{n\vec{\pi}}\Big(x',\Psi(x,a,x',\mu,z),\beta z\Big)  Q^X(dx'|x,\mu^Y,a) \\
   &=& \int_{E_X}\int_{E_Y} \int_{\R_+}  \Eop^{\pi}_{x',y'} \Big[ U\big(s'+ z \sum_{k=0}^{n-1} \beta^{k+1} c(X_k,Y_k,A_k)\big)\Big] \Psi(x,a,x',\mu,z)(dy',ds') Q^X(dx'|x,\mu^Y,a)\\
  &=&  \int_{E_X}\int_{E_Y} \int_{\R_+}  \Eop^{\pi} \Big[ U\big(s'+ z \sum_{k=1}^{n} \beta^{k} c(X_k,Y_k,A_k)\big)\Big|X_1=x',Y_1=y'\Big] \cdot \\ &&
    \hspace*{2cm} \int_{E_Y} \int_{\R_+} q(x',y'|x,y,a) \delta_{s+z c(x,y,a)}(ds')\mu(dy,ds) \nu(dy') \lambda(dx')\\
   &=&\int_{E_Y}\int_{E_Y}\int_{E_X}\int_{\R_+}  \Eop^{\pi} \Big[ U\big(s+z c(x,y,a)+ z \sum_{k=1}^{n} \beta^{k} c(X_k,Y_k,A_k)\big)\Big|X_1=x',Y_1=y'\Big] \\ && \hspace*{2cm} q(x',y'|x,y,a) \mu(dy,ds) \nu(dy') \lambda(dx')\\
   &=&\int_{E_Y}\int_{\R_+}  \Eop^{\pi}_{xy} \Big[ U\big(s+ z \sum_{k=0}^{n} \beta^{k} c(X_k,Y_k,A_k)\big)\Big] \mu(dy,ds)\\
   &=& V_{n+1\pi}(x,\mu,z).
\end{eqnarray*}
and the statement in part a) is shown.

Next we prove parts b) and c) together. From part a) it follows that  for $\pi\in\Pi^M$, the value functions in problem \eqref{eq:V-def} indeed coincide with the value functions of the previously defined MDP. From MDP theory it follows in particular that it is enough to consider Markov policies $\Pi^M$, i.e., $V_n=\inf_{\sigma\in\Pi} V_{n\sigma} = \inf_{\pi\in\Pi^M} V_{n\pi}$ (see e.g. \cite{hin70} Theorem 18.4). Next consider functions $v\in \mathcal{C}({E})$. We show that $Tv\in \mathcal{C}({E})$
and that there exists a minimizer for $v$. Statements b) and c) then follow from Theorem 2.3.8 in \cite{br11}.

We start by proving that $Q^X(\cdot |x,\mu^Y,a)$ is weakly continuous, i.e., we have to show that
\begin{equation}\label{eq:vmap}(x,\mu,a) \mapsto \int v(x') Q^X(dx'|x,\mu^Y,a)\end{equation}
is continuous for all $v\in C_b(E_X)$ where $C_b(E_X)$ is the set of bounded, continuous functions on $E_X$.
Obviously $\mu_n \Rightarrow \mu$ implies that $\mu_n^Y \Rightarrow \mu^Y$ where $\Rightarrow$ denotes weak convergence.
From our standing assumption (A)(v) it follows that $Q(\cdot |x,y,a)$ is weakly continuous. Hence we obtain from Theorem 17.11 in \cite{hin70} that the function in \eqref{eq:vmap} is continuous.


Next we show that
$$(x,a,x',\mu,z) \mapsto \Psi(x,a,x',\mu,z)$$
is continuous at all points where $\Psi$ is defined, i.e., if $(x_n,a_n,x'_n,\mu_n,z_n)$ converges to $(x,a,x',\mu,z)$ in $ E_X\times A\times E_X\times  \Pop_b(E_Y\times \R_+)\times (0,1] $ it follows that $\Psi(x_n,a_n,x'_n,\mu_n,z_n)\Rightarrow \Psi(x,a,x',\mu,z)$ where $(x_n,a_n,x'_n,\mu_n,z_n)$ and $(x,a,x',\mu,z)$ are such that $\int_{E_Y} q^X(x'_n|x_n,y,a_n) \mu^Y(dy)>0$ and $\int_{E_Y} q^X(x'|x,y,a) \mu^Y(dy)>0$. Hence for $v\in C_b(E_Y\times \R_+)$ consider
$$ \int_{E_Y}\int_{\R_+} v(y',s') \Psi(x,a,x',\mu,z)(dy',ds').$$ If we plug in the definition of $\Psi$ we get a quotient whose numerator and denominator will be investigated separately. For the numerator we obtain
\begin{eqnarray*}
   && \int_{E_Y}\int_{\R_+}\int_{E_Y} v\big(y',s+zc(x,y,a)\big) q(x',y'|x,y,a) \nu(dy') \mu(dy,ds)
\end{eqnarray*}
which is continuous by assumption (A)(iv,v) and Theorem 17.11 in \cite{hin70}. The denominator
$$ \int_{E_Y} q^X(x'|x,y,a) \mu^Y(dy)$$
is continuous in $(x,a,x',\mu)$ by the same reasoning. Hence $\Psi$ is continuous.

Now suppose $v\in \mathcal{C}({E})$. Taking into account assumption (A), it obviously follows that $(x,x',a,\mu,z) \mapsto v\Big(x',\Psi(x,a,x',\mu,z),\beta z\Big) $ is lower semicontinuous. Again we apply Theorem 17.11 in \cite{hin70} to obtain that $(x,\mu,z,a) \mapsto \int v\Big(x',\Psi(x,a,x',\mu,z),\beta z\Big) Q^X(dx'|x,\mu^Y,a)$ is lower semicontinuous. Note here that continuity of $\Psi$ at those points where the denominator is positive is sufficient, since the other points form a $Q^X$ null-set.
By Proposition 2.4.3 in \cite{br11} it follows that $(x,\mu,z) \mapsto (Tv)(x,\mu,z)$ is lower semicontinuous and there exists a minimizer of $v$.

The inequality $Tv \ge V_0$ is obtained from
\begin{eqnarray*}
&&    \int_{E_X} v\Big(x',\Psi(x,a,x',\mu,z),\beta z\Big) Q^X(dx'|x,\mu^Y,a) \\
&\ge&   \int_{E_X} \int_{\R_+} U(s') \Psi^Y(x,a,x',\mu,z)(ds') Q^X(dx'|x,\mu^Y,a) \\
   &=& \int_{E_Y}\int_{\R_+} U\big(s+z c(x,y,a)\big) \int_{E_X} q^X(x'|x,y,a) \lambda (dx') \mu(dy,ds)\\
   &\ge& \int_{E_Y}\int_{\R_+} U\big(s) \mu(dy,ds) =V_0(x,\mu,z)
\end{eqnarray*}
which implies the statement.
\end{proof}

\begin{remark}
Note that $\mu\mapsto V_{n\pi}(x,\mu,z)$ is by definition a  linear mapping and thus $\mu\mapsto V_n(x,\mu,z)$ is concave.
\end{remark}

\begin{remark}\label{rem:Top}
Since $V_0\in \mathcal{C}({E}),$ $TV_0\ge V_0$ and since the $T$-operator is monotone,  $V_n =T^nV_0$ is increasing in $n$.
\end{remark}

\begin{remark}\label{rem:reward1}
Of course instead of minimizing cost one could also consider the problem of maximizing reward. Suppose that $r:D\to [\underline{r},\bar{r}]$ (with $0<\underline{r}<\bar{r}$) is a one-stage reward function and the problem is
\begin{equation}\label{eq:prob2}J_N(x) := \sup_{\sigma\in\Pi} \int_{E_Y} \Eop_{xy}^\sigma\Big[U\Big(\sum_{k=0}^{N-1} r(X_k,A_k)\Big)\Big] Q_0(dy),\quad x\in E_X.
\end{equation}
It is possible to treat this problem in exactly the same way using straightforward modifications.
\end{remark}

\section{Some Special Cases}\label{sec:special}

\subsection{The cost function does not depend on the hidden state}\label{subsec:observe}

An important special case is obtained when the one-stage cost function does not depend on the hidden state $y$, i.e., $c(x,y,a)=c(x,a)$. In this case the cost which has accumulated so far is always observable. The recursion for the joint conditional distribution $\mu_n(\cdot|h_n)$ of cost and hidden state simplifies considerable. In order to explain this, we define the operator $\Phi:E_X\times A\times E_X\times \Pop(E_Y) \to \Pop(E_Y)$ by
\begin{equation*}
 \Phi(x,a,x',\mu)(B) :=   \frac{\int_{B}\int_{E_Y} q(x',y'|x,y,a) \mu(dy)\nu(dy')}{\int_{E_Y} q^X(x'|x,y,a) \mu(dy) },\; B\in\mathcal{B}(E_Y).
\end{equation*}
Note that $\Phi$ is exactly the usual updating (Bayesian) operator which appears in classical POMDP (see e.g. \cite{br11}, section 5.2). It updates the conditional probability of the unobservable state. In what follows denote by $(\mu_n^\phi)$ the sequence of probability measures on $E_Y$ generated by $\Phi$ with $\mu_0^\Phi := Q_0$. Then we obtain:

\begin{proposition}
Suppose $c(x,y,a)=c(x,a)$ is independent of $y$. Then $\mu_n(\cdot|h_n)$ from \eqref{eq:murec1} can be written as
\begin{equation}\label{eq:murec2}
    \mu_n(B_1\times B_2|h_n) = \mu_n^Y(B_1|h_n) \cdot \mu_n^S(B_2|h_n), \;\; \mbox{where}\; B_1\times B_2 \in \mathcal{B}(E_Y\times \R_+)
\end{equation}
with $\mu_n^S(\cdot|h_n)= \delta_{\sum_{k=0}^{n-1}\beta^k c(x_k,a_k)}$ and $\mu_n^Y(\cdot|h_n) = \mu_n^\Phi(\cdot|h_n).$
\end{proposition}

\begin{proof}
The proof is by induction on $n$. The statement for $n=0$ is true by definition. Now suppose the statement is true for $n$. We obtain with $h_{n+1}=(h_n,a_n,x')$, $x_n = x$ and $a_n = a$:
\begin{eqnarray*}
\mu_{n+1}(B_1\times B_2|h_{n+1})&=& \frac{\int_{E_Y}\int_{\R_+}\int_{B_1}\int_{B_2} q(x',y'|x,y,a) \nu(dy') \delta_{s+\beta^n c(x,a)}(ds') \mu_n^Y(dy|h_n) \mu_n^S(ds|h_n) }{\int_{E_Y} q^X(x'|x,y,a) \mu^Y_n(dy|h_n)}\\
&=& \frac{\int_{B_1}\int_{E_Y} q(x',y'|x,y,a) \mu_n^Y(dy|h_n)\nu(dy')}{\int_{E_Y} q^X(x'|x,y,a) \mu^Y_n(dy|h_n) } \int_{\R_+} \delta_{s+\beta^n c(x,a)}(B_2) \mu_n^S(ds|h_n) \\
&=&  \Phi\big(x,a,x',\mu_n^Y(\cdot|h_n)\big)(B_1) \cdot  \delta_{\sum_{k=0}^{n}\beta^k c(x_k,a_k)}(B_2).
\end{eqnarray*}
Noting that $\mu_n^Y(\cdot|h_n) = \mu_n^\Phi(\cdot|h_n)$ by the induction hypothesis, the statement follows.
\end{proof}

Thus, the problem simplifies considerably since instead of probability measures on $\mathcal{B}(E_Y\times\R_+)$ we only need to consider probability measures on $\mathcal{B}(E_Y)$ together with an observable sequence of accumulated cost. We can interpret the embedding MDP as one with state space  $E_X\times \Pop(E_Y)\times \R_+\times (0,1]$ and the value iteration reads
\begin{eqnarray*}
  V_0(x,\mu,s,z) &:=& U(s) \\
  V_{n+1}(x,\mu,s,z) &=& \inf_{a\in D(x)} \int V_{n}\Big(x',\Phi(x,a,x',\mu),s+ zc(x,a),\beta z\Big) Q^X(dx'|x,\mu,a),\\
  && \quad\quad\quad \mbox{for}\; (x,\mu,s,z)\in E_X\times \Pop(E_Y)\times \R_+\times (0,1],
\end{eqnarray*}
where $\Phi$ has been defined in the previous calculation.

\begin{remark}
In case there is no unobservable component, i.e., we have a completely observable risk-sensitive MDP, the updating operator $\Psi:E_X\times A\times E_X\times \Pop(\R_+)\times (0,1] \to  \Pop(\R_+)$ boils down to
$$ \Psi(x,a,x',\mu,z)(B) = \int_{B} \delta_{s+zc(x,a)}\mu(ds),\; B\in \mathcal{B}(\R_+)$$
and we obtain $\mu_n(B|h_n)= \delta_{\sum_{k=0}^{n-1} \beta^k c(x_k,a_k)}(B).$
Hence the updating process is deterministic and instead of $\mu$ we can simply store the accumulated cost so far. The value iteration then reads
\begin{eqnarray*}
  V_0(x,s,z) &=& U(s),\quad (x,s,z) \in E_X\times \R_+\times (0,1] \\
  V_{n+1}(x,s,z) &=& \inf_{a\in D(x)} \int V_{n}(x',s+zc(x,a),z\beta) Q(dx'|x,a),
\end{eqnarray*}
which is exactly the situation which has been investigated in \cite{br14}.
\end{remark}

\subsection{A particular class of partially observable control models}
The transition law of the process $(X_n,Y_n)_{n\in\N_0}$ we consider here is quite general. For other general models see Chapter 4 in \cite{herler89}. All these general models contain in particular the following class which appears very often in applications (in particular this is the starting point in \cite{jjr94,dims99}):
\begin{eqnarray*}
  X_{n+1} &=& h(Y_n) + \eta_{n+1} \\
  Y_{n+1} &=& b(Y_n,A_n)+\eta_{n+1}
\end{eqnarray*}
where $(\varepsilon_n)$ is a sequence of independent and identically distributed random variables with density $\varphi_\varepsilon$ and $(\eta_n)$ is a sequence of independent and identically distributed random variables with density $\varphi_\eta$. Both sequences are assumed to be independent and we assume for simplicity that $E_X=E_Y=\R$. We consider here an additive noise but this can also be part of the functions $b$ and $h$ respectively. The transition law under a policy $\pi$ is for $B_1,B_2\in\mathcal{B}(\R)$ given by
\begin{eqnarray*}
  Q(B_1\times B_2|x,y,a) &=& \Pop\big(X_{n+1}\in B_1, Y_{n+1}\in B_2| X_n=x, Y_n=y, A_n=a\big)    \\
   &=&   \Pop\big( h(y) + \eta_{n+1} \in B_1, b(y,a) + \varepsilon_{n+1}\in B_2\big)\\
   &=& \int_{B_1} \varphi_\eta\big(w-h(y)\big)dw \int_{B_2} \varphi_\varepsilon\big(v-b(y,a))\big) dv.
\end{eqnarray*}
According to assumption (A)(v) the resulting density $q$ has to be continuous and bounded in all variables. This is for example satisfied if $b,h$ are continuous and $\varphi_\varepsilon,\varphi_\eta$ are continuous and bounded densities, like e.g. the Gaussian density.

\subsection{Total costs criterion}
In case $\beta=1$, the costs are not discounted and we minimize the utility of the total costs
$$  \sum_{k=0}^{N-1} c(X_k,Y_k,A_k).$$
In this case the $z$-component of the iteration in Theorem \ref{theo:Bellman1} b) does not change. Since in general we start with $z=1$, we can just skip it and obtain the simpler recursion for $n=0,\ldots,N-1$
\begin{eqnarray*}
  V_0(x,\mu)  &:=& \int\int U(s)\mu(dy,ds)  \\
   V_{n+1}(x,\mu) &=&  \inf_{a\in D(x)} \int_{E_X} V_{n}\Big(x',\Psi(x,a,x',\mu)\Big) Q^X(dx'|x,\mu^Y,a),\quad (x,\mu)\in E_X\times \Pop_b(E_Y\times \R_+),
\end{eqnarray*}
where $\Psi(x,a,x',\mu):= \Psi(x,a,x',\mu,1)$ from \eqref{eq:Psi}.
Indeed the $z$-component is equivalent to the knowledge of the time step but since we would like to consider a general problem it makes sense to introduce this component in the model setup in Section \ref{sec:finiteh}.

\subsection{Exponential Utility function}
In this section we assume now that the utility function has the special form $U(x) =  \frac{1}{\gamma} e^{\gamma x}$ with $\gamma \neq 0$.  This situation is often referred to as the usual risk-sensitive problem. Partially observable problems in this setting have already been considered in \cite{whi90,jjr94,hh99,dims99,s04,cchh05}. However still in this case our model is far more general than in the previous literature where the filter is derived with a change of measure technique. As we have shown in \eqref{eq:murec1} such a measure transformation is not needed for the computation of the filter.

Our aim is to specialize the value iteration from Theorem \ref{theo:Bellman1} to this case. In order to do this define for $\mu \in \Pop_b(E_Y\times \R_+)$:
\begin{equation}
    \label{eq:muhat} \hat{\mu} (B) := \frac{\int_{B}\int_{\R_+} e^{\gamma s}\mu(dy,ds)}{\int_{\R_+} e^{\gamma s} \mu^S(ds) },\quad B\in \mathcal{B}(E_Y)
\end{equation}
which obviously yields a new probability measure on $\Pop(E_Y)$.

\begin{remark}\label{rem:informvector}
From Theorem \ref{theo:muinterpr} it follows directly that $\hat{\mu}$ has a certain interpretation. We obtain for $\mu_n$ from Theorem \ref{theo:muinterpr} that
$$ \int_{B}\int_{\R_+} e^{\gamma s}\mu_n(dy,ds|h_n) = \Eop^\pi\Big[1_B(Y_n) \cdot e^{\gamma \sum_{k=0}^{n-1} \beta^k c(X_k,Y_k,A_k)} \Big|h_n\Big].$$
If $\hat{\mu}_n$ is the normalized version of this expression then it coincides with the 'information vector' defined e.g. in \cite{jjr94,cchh05}. Note that we obtain $\hat{\mu}_n$ in a very natural way as a special case of our general $\mu_n$ in Section \ref{sec:finiteh}. 
\end{remark}

Further we can write for $(x,\mu,z)\in E$:
\begin{eqnarray*}
  V_n(x,\mu,z) &=& \int_{\R_+} e^{\gamma s} \inf_{\pi} \frac{1}{\gamma} \int_{E_Y} \Eop_{xy}^{\pi}\Big[ \exp\Big( \gamma z \sum_{k=0}^{n-1} \beta^k c(X_k,Y_k,A_k)\Big)\Big] \mu(dy,ds) \\
  &=& \int_{\R_+} e^{\gamma s} \mu^S(ds) \cdot \inf_{\pi} \frac{1}{\gamma} \int_{E_Y} \Eop_{xy}^{\pi}\Big[\exp\Big( \gamma z \sum_{k=0}^{n-1} \beta^k c(X_k,Y_k,A_k)\Big) \Big] \hat{\mu}(dy)\\
  &=:& \int_{\R_+} e^{\gamma s} \mu^S(ds) \cdot \mathbf{e}_n(x,\hat{\mu},\gamma z).
\end{eqnarray*}
Using this representation, the value iteration in Theorem \ref{theo:Bellman1} can be restricted to the functions $\mathbf{e}_n$.  The state space $E_X\times \Pop(E_Y)\times (0,1]$ is much simpler because measures are only concentrated on $E_Y$.

\begin{theorem}\label{theo:exp}
\begin{itemize}
  \item[a)] For $(x,{\mu},z)\in  E_X\times \Pop(E_Y)\times (0,1]$ it holds that  $\mathbf{e}_0(x,{\mu},\gamma z)= \frac1\gamma$ and for $n=1,\ldots ,N$
  $$ \mathbf{e}_{n+1}(x,{\mu},\gamma z) = \inf_{a\in D(x)} \int_{E_X} \mathbf{e}_{n}\Big(x',{\Psi_e}(x,a,x',{\mu},z),\beta\gamma z\Big) \hat{Q}^X(dx'|x,{\mu},a,\gamma z),$$
  where for $B_1\in \mathcal{B}(E_X), B_2\in\mathcal{B}(E_Y)$
  \begin{eqnarray}
\label{eq:hatQ}   \hat{Q}^X(B_1|x,{\mu},a,z)  &:=&  \int_{B_1}\int_{E_Y} e^{z c(x,y,a)}  q^X(x'|x,{y},a) {\mu}(dy)\lambda(dx'), \\
 \label{eq:hatpsi}  {\Psi_e}(x,a,x',{\mu},z)(B_2)  &:=&  \frac{\int_{B_2}\int_{E_Y} e^{z c(x,y,a)} q(x',y'|x,y,a) {\mu}(dy) \nu(dy')}{\int_{E_Y}\int_{E_Y} e^{z c(x,y,a)} q(x',y'|x,y,a) {\mu}(dy)\nu(dy')}.
  \end{eqnarray}
The value function of \eqref{eq:unobservedMDP} is then given by $J_N(x)=\mathbf{e}_N(x,Q_0,\gamma)$.
  \item[b)] For every $n=1,\ldots ,N$ there exists a  minimizer $f^*_n\in F$ of $\mathbf{e}_{n-1}$ and $(g_0^*,\ldots,g_{N-1}^*)$ with
 $$ g_n^*(h_n) := f_{N-n}^*\big( x_n, {\mu}_n^e(\cdot |h_n),\gamma\beta^n\big),\quad n=0,\ldots,N-1$$
is an optimal policy for problem \eqref{eq:unobservedMDP} where the sequence $(\mu_n^e)$ of posterior distributions is generated by the updating operator $\Psi_e$ with $\mu_0^e:=Q_0$.
  \end{itemize}
\end{theorem}

\begin{proof}
Let $(x,{\mu},z) \in E$.
On one hand we have that
$$ {V}_{n+1}(x,{\mu},z) = \int_{\R_+} e^{\gamma s} \mu^S(ds) \cdot \mathbf{e}_{n+1}(x,\hat{\mu},\gamma z),$$
on the other hand we have by Theorem \ref{theo:Bellman1}:
\begin{eqnarray*}
&&  {V}_{n+1}(x,{\mu},z) = \inf_{a\in D(x)} \int_{E_X} V_{n}\Big(x',\Psi(x,a,x',\mu,z),\beta z\Big) Q^X(dx'|x,\mu^Y,a) \\
   &=& \inf_{a\in D(x)} \int_{E_X} \int_{\R_+} e^{\gamma s'} \Psi^S(x,a,x',\mu,z)(ds')\cdot \mathbf{e}_{n}\Big(x',\hat{\Psi}(x,a,x',{\mu},z),\beta\gamma z\Big) Q^X(dx'|x,\mu^Y,a) \\
   &=& \inf_{a\in D(x)} \int_{E_X}\int_{E_Y}\int_{E_Y}\int_{\R_+} e^{\gamma s+\gamma z c(x,y,a)} q(x',y'|x,y,a) \mu(dy,ds) \nu(dy')\cdot \\ && \hspace*{3cm}\mathbf{e}_{n}\Big(x',\hat{\Psi}(x,a,x',{\mu},z),\beta\gamma z\Big) \lambda(dx') \\
   &=&  \int_{\R_+} e^{\gamma s} \mu^S(ds) \cdot\\
   && \inf_{a\in D(x)} \int_{E_X}\int_{E_Y}\int_{E_Y} e^{\gamma z c(x,y,a)} q(x',y'|x,y,a) \hat{\mu}(dy) \nu(dy') \mathbf{e}_{n}\Big(x',\hat{\Psi}(x,a,x',{\mu},z),\beta\gamma z\Big) \lambda(dx')\\
   &=&  \int_{\R_+} e^{\gamma s} \mu^S(ds) \cdot \inf_{a\in D(x)} \int_{E_X} \mathbf{e}_{n}\Big(x',\hat{\Psi}(x,a,x',{\mu},z),\beta\gamma z\Big) \hat{Q}^X(dx'|x,\hat{\mu},a,\gamma z).
\end{eqnarray*}
It remains to show that $ \hat{\Psi}(x,a,x',{\mu},z)={\Psi_e}(x,a,x',\hat{\mu},\gamma z)$ which is defined in \eqref{eq:hatpsi}.
We obtain for $B\in\mathcal{B}(E_Y)$:
\begin{eqnarray*}
 \hat{\Psi}(x,a,x',{\mu},z)(B)   &=& \frac{\int_{B}\int_{\R_+} e^{\gamma s'} {\Psi}(x,a,x',\mu,z)(dy',ds') }{\int_{E_Y}\int_{\R_+} e^{\gamma s'} {\Psi}(x,a,x',\mu,z)(dy',ds') } \\
   &=& \frac{\int_{B}\int_{E_Y}\int_{\R_+} q(x',y'|x,y,a)  e^{\gamma s+\gamma z c(x,y,a)} \mu(dy,ds) \nu(dy') }{\int_{E_Y}\int_{E_Y}\int_{\R_+} q(x',y'|x,y,a)  e^{\gamma s+\gamma z c(x,y,a)} \mu(dy,ds) \nu(dy')} \\
   &=&  \frac{\int_{B}\int_{E_Y} q(x',y'|x,y,a)  e^{\gamma z c(x,y,a)} \hat{\mu}(dy) \nu(dy') }{\int_{E_Y}\int_{E_Y} q(x',y'|x,y,a)  e^{\gamma z c(x,y,a)} \hat{\mu}(dy) \nu(dy')}\\
   &=& {\Psi_e}(x,a,x',\hat{\mu},\gamma z)(B).
\end{eqnarray*}
Hence part a) is shown. Part b) follows as in Theorem \ref{theo:Bellman1} c).
\end{proof}

\begin{remark}\label{rem:psie}
If $(\mu_n)$ is generated by $\Psi$ with $\mu_0:=Q_0\otimes \delta_0$ (note that $\mu_n$ are probability measures on $\mathcal{B}(E_Y\times \R_+)$), then
$\hat{\mu}(\cdot |h_n)= \mu_n^e(\cdot |h_n)$, i.e., $(\mu_n^e)$ is the sequence of information vectors (see Remark  \eqref{rem:informvector}). The statement follows directly from the proof of the previous theorem.
\end{remark}

\subsection{Power Utility function}
In this section we assume that the utility function has the special form $U(x) = \frac{1}{\gamma}x^\gamma$ with $\gamma\neq 0$. Thus, we obtain:
\begin{eqnarray*}
&&  V_{n}(x,\mu,z) = \inf_{\pi} \frac{1}{\gamma}\int_{E_Y}\int_{\R_+} \Eop^\pi_{xy} \Big[\big( s+ z \sum_{k=0}^{n-1} \beta^k c(X_k,Y_k,A_k)\big)^\gamma\Big]\mu(dy,ds)\\
  &=& z^\gamma \inf_{\pi} \frac{1}{\gamma} \int_{E_Y}\int_{\R_+} \Eop^\pi_{xy} \Big[\big( \frac{s}{z}+ \sum_{k=0}^{n-1} \beta^k c(X_k,Y_k,A_k)\big)^\gamma\Big]\mu(dy,ds)\\
  &=& z^\gamma \inf_{\pi} \frac{1}{\gamma} \int_{E_Y}\int_{\R_+} \Eop^\pi_{xy} \Big[\big( \tilde{s}+ \sum_{k=0}^{n-1} \beta^k c(X_k,Y_k,A_k)\big)^\gamma\Big]\tilde{\mu}(dy,d\tilde{s})\\
  &=:& z^\gamma d_{n}(x,\tilde{\mu}),
\end{eqnarray*}
where $\tilde{\mu}$ is defined by $\tilde{\mu}(B_1\times B_2) := \mu( B_1 \times \frac{1}z B_2)$ for $B_1\times B_2\in \mathcal{B}(E_Y\times\R_+)$. Hence $\tilde{\mu}\in \Pop_b(E_Y\times \R_+)$.

\begin{theorem}\label{theo:power}
\begin{itemize}
  \item[a)] For $(x,{\mu})\in E_X\times \Pop_b(E_Y\times \R_+)$ it holds  $d_0(x,{\mu}) := \frac1\gamma \int\int s^\gamma \mu(dy,ds)$ and for $n=1,\ldots ,N$
  $$ d_{n+1}(x,{\mu}) = \inf_{a\in D(x)} \beta^\gamma \int_{E_X} d_{n}\Big(x',{\Psi_p}(x,a,x',{\mu})\Big) {Q}^X(dx'|x,{\mu},a),$$
  where for $B \in \mathcal{B}(E_Y\times\R_+)$
  \begin{eqnarray*}
   {\Psi_p}(x,a,x',{\mu})(B)  &:=&  \frac{\int_{E_Y}\int_{\R_+}\Big(\int_{B} q(x',y'|x,y,a) \nu(dy')\delta_{\frac{s+c(x,y,a)}{\beta}}(ds') \Big) {\mu}(dy,ds) }{\int_{E_Y}  q^X(x'|x,y,a) {\mu}^Y(dy) }.
  \end{eqnarray*}
The value function of \eqref{eq:unobservedMDP} is then  given by $J_N(x)=d_N(x,Q_0\otimes \delta_0)$.
  \item[b)] For every $n=1,\ldots ,N$ there exists a  minimizer $f^*_n\in F$ of $d_{n-1}$ and $(g_0^*,\ldots,g_{N-1}^*)$ with
 $$ g_n^*(h_n) := f_{N-n}^*\big( x_n, {\mu}_n^p(\cdot |h_n)\big),\quad n=0,\ldots,N-1$$
is an optimal policy for problem \eqref{eq:unobservedMDP},  where the sequence $(\mu_n^p)$ is generated by $\Psi_p$ with $\mu_0^p:=Q_0\otimes \delta_0$.
  \end{itemize}
\end{theorem}

\begin{proof}
On one hand we have shown
$$  V_{n+1}(x,\mu,z) = z^\gamma d_{n+1}(x,\tilde{\mu}).$$
On the other hand we obtain with Theorem \ref{theo:Bellman1}
\begin{eqnarray*}
  &&{V}_{n+1}(x,{\mu},z) = \inf_{a\in D(x)}\int_{E_X} V_{n}\Big(x',\Psi(x,a,x',\mu,z),\beta z\Big) Q^X(dx'|x,\mu^Y,a)\\
   &=& \inf_{a\in D(x)} \beta^\gamma z^\gamma \int_{E_X} d_{n}\Big(x',\tilde{\Psi}(x,a,x',\mu,z),\beta z\Big) Q^X(dx'|x,\mu^Y,a).
\end{eqnarray*}
It remains to show that $\tilde{\Psi}(x,a,x',\mu,z) = {\Psi_p}(x,a,x',\tilde{\mu})$.

Here we obtain for $B\in\mathcal{B}(E_Y\times\R_+)$:
\begin{eqnarray*}
   \tilde{\Psi}(x,a,x',{\mu},z)(B)  &=&  \frac{\int_{E_Y}\int_{\R_+}\Big(\int_{B} q(x',y'|x,y,a) \nu(dy')  \delta_{\frac{\frac{s}{z}+c(x,y,a)}{\beta}}(ds')\Big) \mu(dy,ds)  }{\int_{E_Y}\int_{\R_+}\int_{\R_+}\int_{E_Y} q(x',y'|x,y,a) \nu(dy') \delta_{\frac{\frac{s}{z}+c(x,y,a)}{\beta}}(ds') \mu(dy,ds) } \\
   &=&  \frac{ \int_{E_Y}\int_{\R_+}\Big(\int_{B} q(x',y'|x,y,a)  \nu(dy')\delta_{\frac{\tilde{s}+c(x,y,a)}{\beta}}(ds')\Big) \tilde{\mu}(dy,d\tilde{s})  }{\int_{E_Y}\int_{\R_+}\int_{\R_+} \int_{E_Y} q(x',y'|x,y,a)  \nu(dy') \delta_{\frac{\tilde{s}+c(x,y,a)}{\beta}}(ds') \tilde{\mu}(dy,d\tilde{s})}\\
   &=& {\Psi_p}(x,a,x',\tilde{\mu})(B).
\end{eqnarray*}
Hence part a) is shown. Part b) follows as in Theorem \ref{theo:Bellman1} c).
\end{proof}

\begin{remark}
If $(\mu_n)$ is generated by $\Psi$ with $\mu_0:=Q_0\otimes \delta_0$, then
$\tilde{\mu}(\cdot |h_n)= \mu_n^p(\cdot |h_n)$. The statement follows directly from the proof of the previous theorem.

\end{remark}

\begin{remark}
Note that the special case $U(x)=\log(x)$ can be treated similar. It can also be obtained from the power utility case by letting $\gamma\to 0$.
\end{remark}

\begin{remark}
Also the updating operators $\Psi_e$ and $\Psi_p$ simplify considerably if the cost function $c(x,y,a)$ is independent of $y$ (see Section \ref{subsec:observe}).
\end{remark}

\section{Application: Risk-Sensitive Bayesian House Selling Problem}\label{sec:casino}
As an application we consider a risk-sensitive Bayesian extension of the classical house selling problem with finite time horizon. We assume that offers for a house $X_0,\ldots,X_N$ arrive independently and are identically distributed with distribution $Q_\theta$.  Here $\theta\in \Theta$ is an unknown parameter and $\Theta$ is assumed to be a Borel space. Further we assume that $Q_\theta$ has a $\lambda$-density $q(x|\theta)$ which is continuous in both parameters with compact support. A prior distribution $Q_0$ for $\theta$ is given. As long as offers are rejected an observation cost of $c_\theta>0$ has to be paid which also depends on $\theta$ and cannot be observed. We suppose that $c_\theta$ is continuous in $\theta$. When an offer is accepted, the price is obtained and the process ends. If one has not stopped before $N$, the last offer has to be accepted. The aim is to find the maximal risk-sensitive stopping reward
\begin{equation}\label{BReq:H}
    J_N(x) := \sup_{0\le \tau \le N}  \int_{\Theta} \Eop_{x\theta}\Big[U\Big(X_{\tau}-c_\theta \tau \Big)   \Big] Q_0(d\theta)
\end{equation}
where the supremum is taken over all stopping times $\tau$. Here we assume that $U:\R\to\R$ is strictly increasing and concave. In order to have a well-defined problem we also assume that $\sup_\theta \Eop_\theta[X_1^+]<\infty$. This risk-sensitive Bayesian house selling problem can be solved in a similar way as our general model with $Y_n\equiv \theta$ and $E_Y=\Theta$, i.e., the unobservable component is simply the unknown parameter and $c(x,\theta)=c_\theta$ (independent of $x$). However note that we also have a terminal reward in case we have not stopped before which equals the last offer. Risk-sensitive house selling problems with complete observation have been treated in \cite{m00}. A  risk-sensitive Bayesian house selling problem has been considered in \cite{br15} however with fixed observation costs $c$ (independent of $\theta$). We define the updating operator $\Psi$ for the joint conditional probability of the unknown parameter $\theta$ and the accumulated cost so far only in case we do not stop because otherwise the problem ends immediately. Also note that since $\beta=1$ we can skip the $z$-component in the state space. Moreover, the i.i.d. assumption on the offers implies that $\Psi$ does not depend on $x$ which is the previous offer. The updating operator is given by
$$ \Psi(x',\mu) (B_1\times B_2) = \frac{\int_{B_1}\int_{\R_-} q(x'|\theta) \delta_{s-c_\theta}(B_2) \mu(d\theta,ds)}{\int_\Theta q(x'|\theta) \mu^\Theta(d\theta)},\quad B_1\times B_2\in\mathcal{B}(\Theta\times\R_-).$$
According to Theorem \ref{theo:Bellman1} we obtain $J_N$ by computing the functions $V_n$. These are given by
 \begin{eqnarray*}
    V_0(x,\mu) &=& \int U\big(x+s\big)\mu^S(ds) =: U_\mu(x) \\
    V_n(x,\mu) &=& \max\Big\{ U_\mu(x), d_n(\mu) \Big\}
  \end{eqnarray*}
with $ d_n(\mu):=\int_{\R} V_{n-1}\Big(x',\Psi(x',\mu)\Big) Q^X(dx'|\mu^\Theta)$. We have that $J_N(x)=V_N(x,Q_0\otimes \delta_0)$. Note that $Q^X(dx'|\mu^\Theta)$ is given by
$$ Q^X(B|\mu^\Theta) = \int_B\int_\Theta q(x'|\theta) \mu^\Theta(d\theta)\lambda(dx'),\quad B\in\mathcal{B}(E_X).$$ When we define $f_n^*(x,\mu)=stop$ if $U_\mu(x)\ge d_n(\mu)$ and $(g_0^*,\ldots,g_{N-1}^*)$ by
 $$ g_n^*(h_n) := f_{N-n}^*\big( x_n, \mu_n(\cdot|h_n)\big),\quad h_n=(x_0,x_1,\ldots,x_n),\quad n=0,\ldots,N-1,$$
then the optimal stopping time for problem \eqref{BReq:H} is given by
 $$ \tau^* := \inf\{ n\in\N_0 : g_n^*(h_n) = stop\}\wedge N.$$
 Let us now further investigate the optimal stopping time $\tau^*$. As in Section \ref{sec:finiteh} we define by $\mu_n(\cdot| h_n)$ the sequence of conditional probabilities generated by the updating-operator. Then we have
 $$ g_n^*(h_n)= stop \quad \Leftrightarrow \quad U_{\mu_n(\cdot|h_n)}(x_n) \ge d_{N-n}(\mu_n(\cdot|h_n)).$$
Since $x\mapsto U_\mu(x)$ is increasing and continuous, the inverse function $U^{-1}_{\mu}$ exists and we obtain
 $$ g_n^*(h_n)= stop \quad \Leftrightarrow \quad x_n \ge U_{\mu_n(\cdot|h_n)}^{-1}\big(d_{N-n}(\mu_n(\cdot|h_n))\big) =: x_{n,N}^*(\mu_n(\cdot|h_n)).$$
We call $x_{n,N}^*(\cdot)$ {\em reservation level}. The reservation levels depend on $\mu_n$ and $U$. The optimal stopping time is hence the first time, the offer exceeds the corresponding, history dependent reservation level.

\begin{theorem}\label{BRtheo:xrec}
\begin{itemize}
  \item[a)] The optimal stopping time for the risk-sensitive Bayesian house selling problem is given by
$$\tau^* = \inf\big\{ n\in\N_0 : X_n \ge x_{n,N}^*(\mu_n(\cdot|h_n))\big\}\wedge N.$$
  \item[b)] The reservation levels can recursively be computed by
  \begin{eqnarray*}
    x_{N-1,N}^*(\mu_{N-1}) &=& U_{\mu_{N-1}}^{-1} \circ \int_\R  \int_\R U(x+s) \mu^S_{N-1}(ds) Q^X(dx|\mu^\Theta_{N-1})   \\
    x_{n,N}^*(\mu_n) &=&  U_{\mu_n}^{-1} \circ \int_\R U_{\Psi(x,\mu_n)}\Big(\max\big\{ x, x_{n+1,N}^*(\Psi(x,\mu_n))\big\}\Big)Q^X(dx|\mu_n^\Theta).
  \end{eqnarray*}
\end{itemize}
\end{theorem}

\begin{proof}
Part a) is clear from the definition and the previous results. Part b) can be shown by inserting the correct definitions. For $n=N-1$ we obtain from the definition of $x_{N-1,N}^*(\mu) $ that
$$x_{N-1,N}^*(\mu) = U_{\mu}^{-1} \big( d_1(\mu)\big)$$ with $$ d_1(\mu)= \int V_0\big(x,\Psi(x,\mu)\big) Q^X(dx|\mu^\Theta).$$
For $ x_{n,N}^*$ we obtain by definition:
$$ x_{n,N}^*(\mu)= U_\mu^{-1}\big(d_{N-n}(\mu)\big).$$
Further $d_{N-n}(\mu_n)$ can be written as
\begin{eqnarray*}
&&  d_{N-n}(\mu_n) =  \int_{\R} V_{N-n-1}\Big(x,\Psi(x,\mu_n),\Big) Q^X(dx|\mu_n^\Theta)\\
   &=& \int_{\R} \max\Big\{ U_{\Psi(x,\mu_n)}(x), d_{N-n-1}(\Psi(x,\mu_n))\Big\}Q^X(dx|\mu_n^\Theta)\\
 &=& \int_{\R} U_{\Psi(x,\mu_n)}\Big(\max\big\{ x, U_{\Psi(x,\mu_n)}^{-1}\circ d_{N-n-1}(\Psi(x,\mu_n))\big\}\Big)Q^X(dx|\mu_n^\Theta)
\end{eqnarray*}
and the statement follows from the definition of $x_{n,N}^*$.
\end{proof}

\section{Infinite Horizon Problems}
Here we consider an infinite time horizon and  $\beta \in (0,1)$, i.e., we are interested in
\begin{equation}\label{eq:Jinfty}
    J_\infty(x) := \inf_{\sigma\in\Pi} \int_{E_Y}\Eop_{xy}^\sigma\Big[U\Big(\sum_{k=0}^{\infty} \beta^k c(X_k,Y_k,A_k)\Big)\Big] Q_0(dy),\quad x\in E.
\end{equation}
We will consider concave and convex utility functions separately.

\subsection{Concave Utility Function}
We first investigate the case of a concave utility function $U:\R_+\to\R$. This situation represents a risk seeking decision maker.

In this subsection we use the following notations
\begin{eqnarray}
 \nonumber V_{\infty\sigma}(x,\mu,z) &:=&  \int_{E_Y}\int_{\R_+}\Eop^\sigma_{xy} \Big[U\Big(s+z\sum_{k=0}^\infty \beta^k c(X_k,Y_k,A_k)\Big)\Big]\mu(ds,dy),\\
\label{eq:prob4}  V_\infty(x,\mu,z) &:=& \inf_{\sigma\in\Pi} V_{\infty\sigma}(x,\mu,z),\quad \quad (x,\mu,z)\in {E}.\end{eqnarray}
We are interested in obtaining $V_\infty(x,Q_0\otimes \delta_0,1)=J_\infty(x)$. For a stationary policy $\pi=(f,f,\ldots)\in\Pi^M$ we write $V_{\infty\pi}=V_{f}$ and denote 
\begin{eqnarray*}
  \bar{b}(\mu,z) &:=& \int_{\R_+} U\Big(s+\frac{z\bar{c}}{1-\beta}\Big)\mu^S(ds),\\
  \underline{b}(\mu,z) &:=& \int_{\R_+} U\Big(s+\frac{z\underline{c}}{1-\beta}\Big)\mu^S(ds),\quad (\mu,z)\in\Pop_b(E_Y\times \R_+)\times[0,1].
\end{eqnarray*}

Then we obtain the main theorem of this section:

\begin{theorem}\label{theo:limit2}
The following statements hold true:
\begin{itemize}
  \item[a)] $V_\infty$ is the unique solution of $v= T v$ in $\mathcal{C}({E})$ with $\underline{b}(\mu,z)\le v(x,\mu,z)\le \bar{b}(\mu,z)$ for $T$ defined in \eqref{eq:Top}. Moreover, $T^n V_0 \uparrow V_\infty, T^n \underline{b}\uparrow V_\infty$ and $T^n \bar{b} \downarrow V_\infty$ for $n\to\infty$. The value function of \eqref{eq:Jinfty} is given by $J_\infty(x) = V_\infty(x,Q_0\otimes\delta_0,1)$.
  \item[b)] There exists a minimizer $f^*$ of $V_\infty$ and $(g_0^*,g_1^*,\ldots )$ with
 $$ g_n^*(h_n) := f^*\Big(x_n, \mu_n(\cdot|h_n),\beta^n \Big)$$ is an optimal policy for \eqref{eq:Jinfty}.
\end{itemize}
\end{theorem}

\begin{proof}
\begin{itemize}
 \item[a)]
 We first show that $V_n=T^n V_0 \uparrow V_\infty$ for $n\to\infty$.
  To this end note that for $U:\R_+\to\R$ increasing and concave we obtain the inequality
$$ U(s_1+s_2) \le U(s_1)+U'_-(s_1) s_2,\quad s_1,s_2\ge 0$$
where $U'_-$ is the left-hand side derivative of $U$ which exists since $U$ in concave. Moreover, $U'_-(s) \ge 0$ and $U'_-$ is non-increasing. For $(x,\mu,z)\in {E}$ and $\sigma\in\Pi$ it holds
\begin{eqnarray}
\nonumber  && V_n(x,\mu,z)\le V_{n\sigma}(x,\mu,z) \le V_{\infty\sigma}(x,\mu,z)  \\
\nonumber  &=& \int_{E_Y}\int_{\R_+}\Eop^\sigma_{xy} \Big[U\Big(s+z\sum_{k=0}^\infty \beta^k c(X_k,Y_k,A_k)\Big)\Big]\mu(dy,ds) \\
\nonumber    &\le&  \int_{E_Y}\int_{\R_+} \Eop_{xy}^\sigma \Big[U\Big(s+z\sum_{k=0}^{n-1} \beta^k c(X_k,Y_k,A_k)\Big)\Big]\mu(dy,ds) \\
\nonumber    && +  \int_{E_Y} \int_{\R_+}\Eop_{xy}^\sigma \Big[U'_-\Big(s+z\sum_{k=0}^{n-1} \beta^k c(X_k,Y_k,A_k)\Big) z \sum_{m=n}^{\infty} \beta^k c(X_m,Y_m,A_m) \Big]\mu(dy,ds)\\
\nonumber &\le & V_{n\sigma}(x,\mu,z) + \beta^n \frac{z\bar{c}}{1-\beta} \int_{\R_+} U'_-(s+z\underline{c})\mu^S(ds)  \\
\label{eq:epsinequ}&\le&  V_{n\sigma}(x,\mu,z) + \beta^n \frac{z\bar{c}}{1-\beta} U'_-(z\underline{c})   =:  V_{n\sigma}(x,\mu,z) + \varepsilon_n(z),
\end{eqnarray}
where $\varepsilon_n(z)$ has implicitly been defined in the last equation.

Obviously $\lim_{n\to\infty} \varepsilon_n(z) =0$. Taking the infimum over all policies in the preceding inequality yields:
$$ V_n(x,\mu,z) \le V_\infty(x,\mu,z) \le V_n(x,\mu,z) + \varepsilon_n(z).$$
Letting $n\to\infty$ yields $V_n=T^n V_0 \uparrow V_\infty$ for $n\to\infty$. Note that the convergence of $T^n V_0$ is monotone (see Remark \ref{rem:Top}).

By direct inspection we obtain $\underline{b}\le V_\infty\le \bar{b}$. We next show that $V_\infty = TV_\infty$. Note that $V_n\le V_\infty$ for all $n$. Since $T$ is increasing we have $V_{n+1} = TV_n \le TV_\infty$ for all $n$. Letting $n\to\infty$ implies $V_\infty\le TV_\infty$. For the reverse inequality recall that  $V_n+\varepsilon_n \ge V_\infty$ from \eqref{eq:epsinequ}. Applying the $T$-operator yields $V_{n+1}+ \varepsilon_{n+1}= T(V_n+\varepsilon_n) \ge TV_\infty$ and letting $n\to\infty$ we obtain $V_\infty\ge TV_\infty$. Hence it follows $V_\infty= TV_\infty$.

Next, we obtain \begin{eqnarray*} (T \bar{b})(\mu,z)  &=& \inf_{a\in D(x)} \int_{\R_+} U\Big(s'+\frac{z\beta \bar{c}}{1-\beta}\Big) \Psi^S(x,a,x'\mu,z)(ds') \\
&\le& \int_{\R_+} U\Big(s+z \bar{c}+\frac{z\beta\bar{c}}{1-\beta}\Big)\mu^S(ds)\\
&= & \int_{\R_+}U\Big(s+\frac{z\bar{c}}{1-\beta}\Big)\mu^S(ds) = \bar{b}(\mu,z).
\end{eqnarray*}
Analogously $T\underline{b}\ge \underline{b}$. Thus we get that $T^n \bar{b}\downarrow $ and $T^n \underline{b}\uparrow $ and the limits exist. Moreover, we obtain by iteration:
\begin{eqnarray*}
  &&(T^n \underline{b})(x,\mu,z) =\\
  &=& \inf_{\pi\in\Pi^M} \int_{E_Y}\int_{\R_+} \Eop_{xy}^\pi \Big[ U\Big(s+\frac{z\underline{c}\beta^n}{1-\beta} +z\sum_{k=0}^{n-1} \beta^k c(X_k,Y_k,A_k)\Big)\Big]\mu(dy,ds) \ge (T^n V_0)(x,\mu,z). \\
   &&(T^n \bar{b})(x,\mu,z)  =\\
     &=&\inf_{\pi\in\Pi^M}  \int_{E_Y}\int_{\R_+} \Eop_{xy}^\pi \Big[ U\Big(s+ \frac{z\bar{c}\beta^n}{1-\beta} +z \sum_{k=0}^{n-1} \beta^k c(X_k,Y_k,A_k) \Big)\Big]\mu(dy,ds).
\end{eqnarray*}
Using $U(s_1+s_2)-U(s_1)\le U'_-(s_1) s_2$ we obtain:
\begin{eqnarray*}
   0&\le& (T^n \bar{b})(x,\mu,z) -  (T^n \underline{b})(x,\mu,z) \le(T^n \bar{b})(x,\mu,z) -  (T^n V_0)(x,\mu,z)  \\ &\le& \sup_{\pi\in\Pi}  \int_{E_Y}\int_{\R_+} \Eop_{xy}^\pi \Big[ U\Big(s+\frac{z\bar{c}\beta^n}{1-\beta} +z \sum_{k=0}^{n-1} \beta^k c(X_k,Y_k,A_k)\Big)-\\
&& \hspace*{4cm}
U\Big(s+z\sum_{k=0}^{n-1} \beta^k c(X_k,Y_k,A_k)  \Big) \Big] \mu (dy,ds) \\
   &\le&  \varepsilon_n(z)
\end{eqnarray*}
and the right-hand side converges to zero for $n\to\infty$. As a result $T^n \bar{b} \downarrow V_\infty$ and $T^n \underline{b} \uparrow V_\infty$ for $n\to\infty$.

Since $V_n$ is lower semicontinuous, this yields immediately that $V_\infty$ is again lower semicontinuous, thus $V_\infty \in \mathcal{C}({E})$.

For the uniqueness suppose that $v\in \mathcal{C}({E})$ is another solution of $v=Tv$ with \linebreak $\underline{b}\le v\le \bar{b}$. Then $T^n \underline{b} \le v \le T^n \bar{b}$ for all $n\in\N$ and since the limit $n\to\infty$ of the right and left-hand side are equal to $V_\infty$ the statement follows.

  \item[b)]
 The existence of a minimizer follows from our standing assumption (A) as in the proof of Theorem \ref{theo:Bellman1}. From our assumption and the fact that $V_\infty\ge V_0$ we obtain
  $$V_\infty = \lim_{n\to\infty} T_{f^*}^n V_\infty \ge \lim_{n\to\infty} T_{f^*}^n V_0 = \lim_{n\to\infty} V_{n(f^*,f^*,\ldots)} = V_{f^*}\ge V_\infty$$
  where the last equation follows with dominated convergence. Hence $(g_0^*,g_1^*,\ldots )$ is optimal for  \eqref{eq:Jinfty}.
\end{itemize}
\end{proof}

\subsection{Convex Utility Function}
Here we consider the problem with convex utility $U$. This situation represents a risk averse decision maker. The value functions $V_{n\sigma},V_n,V_{\infty\sigma},V_\infty$ are defined as in the previous section.

\begin{theorem}\label{theo:limit4}
Theorem \ref{theo:limit2} also holds for convex $U$.
\end{theorem}

\begin{proof}
The proof follows along the same lines as in Theorem \ref{theo:limit2}. The only difference is that we have to use another inequality: Note that for $U:\R_+\to\R$ increasing and convex we obtain the inequality
$$ U(s_1+s_2) \le U(s_1)+U'_+(s_1+s_2) s_2,\quad s_1,s_2\ge 0$$
where $U'_+$ is the right-hand side derivative of $U$ which exists since $U$ in convex. Moreover, $U'_+(s) \ge 0$ and $U'_+$ is increasing.  Thus, we obtain for $(x,\mu,z)\in {E}$ and $\sigma\in\Pi$:
\begin{eqnarray*}
  V_n(x,\mu,z)&\le& V_{n\sigma}(x,\mu,z) \le V_{\infty\sigma}(x,\mu,z) \\
  &=& \int_{E_Y}\int_{\R_+}\Eop^\sigma_{xy} \Big[U\Big(s+z\sum_{k=0}^\infty \beta^k c(X_k,Y_k,A_k)\Big)\Big]\mu(dy,ds) \\
 &\le&  \int_{E_Y}\int_{\R_+} \Eop_{xy}^\sigma  \Big[U\Big(s+z\sum_{k=0}^{n-1} \beta^k c(X_k,Y_k,A_k)\Big)\Big] +\\
  && \hspace*{2cm}+\Eop_{xy}^\sigma \Big[U'_+\Big( s+z\sum_{k=0}^\infty \beta^k c(X_k,Y_k,A_k)\Big) z \sum_{k=n}^\infty \beta^{k}c(X_k,Y_k,A_k)\Big]\mu(dy,ds)  \\
   &\le&  V_{n\sigma}(x,\mu,z) + \frac{z\bar{c}\beta^n }{1-\beta} \int_{\R_+} U'_+\Big( s+\frac{z\bar{c}}{1-\beta}\Big)\mu^S(ds) .
\end{eqnarray*}
Note that the last inequality follows from the fact that $c$ is bounded from above by $\bar{c}$.
Now denote $\delta_n(\mu,z) := \frac{z\bar{c}\beta^n }{1-\beta} \int_{\R_+}U'_+\Big(s+ \frac{z\bar{c}}{1-\beta}+\Big)\mu^S(ds)  $. Obviously $\lim_{n\to\infty} \delta_n(\mu,z) =0$. Taking the infimum over all policies in the above inequality yields:
$$ V_n(x,\mu,z) \le V_\infty(x,\mu,z) \le V_n(x,\mu,z) +\delta_n(\mu,z).$$
Letting $n\to\infty$ yields $T^n V_0\to V_\infty$.

Further we have to use the inequality
\begin{eqnarray*}
   0&\le& (T^n \bar{b})(x,\mu,z) -  (T^n \underline{b})(x,\mu,z)\le (T^n \bar{b})(x,\mu,z) -  (T^n V_0)(x,\mu,z)  \\ &\le& \sup_{\pi\in\Pi} \int_{E_Y}\int_{\R_+}\Eop_x^\pi \Big[ U\Big(s+\frac{z\bar{c}\beta^n}{1-\beta} +z \sum_{k=0}^{n-1} \beta^k c(X_k,Y_k,A_k)\Big)-\\
&& \hspace*{2cm}
U\Big(s+z\sum_{k=0}^{n-1} \beta^k c(X_k,Y_k,A_k)  \Big) \Big]\mu(dy,ds) \\
   &\le&  \frac{z\bar{c}\beta^n}{1-\beta} \int_{\R_+}U'_+ \Big(s+  \frac{z\bar{c}}{1-\beta} \Big)\mu^S(ds) = \delta_n(\mu,z)
\end{eqnarray*}
and the right-hand side converges to zero for $n\to\infty$.
\end{proof}

\subsection{Exponential Utility}
Of course the result for the infinite horizon problem can now be specialized to various situations like in Section \ref{sec:special}. This can be done rather straightforward. We only present the case of the exponential utility due to its importance.

\begin{corollary}
 In case $U(x)= \frac 1\gamma e^{\gamma x}$ with $\gamma\neq 0$, we obtain
\begin{itemize}
  \item[a)]$V_\infty(x,\mu,z) = \int e^{\gamma s}\mu^S(ds)\cdot \mathbf{e}_\infty(x,\hat{\mu},\gamma z),\; (x,\mu,z)\in E_X\times\Pop(E_Y\times\R_+)\times(0,1]$ where $\hat{\mu}$ has  been defined in \eqref{eq:muhat} and the function $\mathbf{e}_\infty$ is the unique fixed point of
$$\mathbf{e}_\infty(x,{\mu},\gamma z) = \inf_{a\in D(x)}   \int_{E_X} \mathbf{e}_{\infty}(x',\Psi_e(x,a,x',{\mu},\gamma z),\beta\gamma z) \hat{Q}^X\big(dx' | x,{\mu},a,\gamma z\big),$$
for $(x,\mu,z)\in E_X\times \Pop(E_Y)\times (0,1]$ with $U(\frac{z\underline{c}}{1-\beta}) \le \mathbf{e}_\infty(x,\mu,\gamma z) \le U(\frac{z\bar{c}}{1-\beta})$. The value function of \eqref{eq:Jinfty} is then given by $J_\infty(x) = \mathbf{e}_\infty(x,Q_0,\gamma).$
  \item[b)] There exists a minimizer $f^*$ of $ \mathbf{e}_\infty$ and $(g_0^*,g_1^*,\ldots )$ with
 $$ g_n^*(h_n) := f^*\Big(x_n, \mu_n^e(\cdot|h_n),\gamma\beta^n \Big)$$ is an optimal policy for \eqref{eq:Jinfty}, where the sequence $(\mu_n^e)$ of posterior distributions is generated by the updating operator $\Psi_e$ with $\mu_0^e:=Q_0$ like in Theorem \ref{theo:exp}.

\end{itemize}

\end{corollary}

{\bf Acknowledgements:} The authors would like to thank three referees for helpful comments and suggestions which improved the presentation of the paper.

\bibliographystyle{plainnat}

\end{document}